\documentclass[12pt,a4paper]{amsart}
\usepackage{latexsym}
\usepackage{amsfonts}
\usepackage[mathscr]{eucal}
\usepackage{epsfig}
\usepackage{a4wide}

\def\@copyright{}
\newtheorem{definition}{Definition}[section]

\newtheorem{theorem}{Theorem}[section]

\newtheorem{example}{Example}[section]

\newtheorem{lemma}{Lemma}[section]

\newtheorem{remark}{Remark}[section]

\newcommand{\N}{ \mathbb{N} }

\newcommand{\R}{ \mathbb{R} }

\newcommand{\calF}{\mathcal{F}}

\newcommand{\calH}{\mathcal{H}}
\newcommand{\calK}{\mathcal{K}}

\newcommand{\calR}{\mathcal{R}}

\newcommand{\calS}{\mathcal{S}}
\newcommand{\calT}{\mathcal{T}}

\newcommand{\eins}{{\mathbf 1}}

\begin{document}
\bibliographystyle{chicago}

\pagestyle{plain}
\begin{titlepage}
\pagenumbering{arabic}
\title[???]{}
\date{September 2003}
\end{titlepage}

\maketitle

\begin{center}
  \Large
  OPTIMAL SEQUENTIAL KERNEL DETECTION
  FOR DEPENDENT PROCESSES\\[1cm]
\end{center}
\vskip 1cm
\begin{center}
  Ansgar Steland\footnote{Address of correspondence: Ansgar Steland, Ruhr-Universit\"at Bochum,
Fakult\"at f\"ur Mathematik, Mathematik 3 NA 3/71, Universit\"atsstr. 150,
D-44780 Bochum, Germany.}\\
  Fakult\"at f\"ur Mathematik\\
  Ruhr-Universit\"at Bochum, Germany\\
  ansgar.steland@ruhr-uni-bochum.de
\end{center}

\newpage

\begin{abstract}
In many applications one is interested to detect certain
(known) patterns in the mean of a process with smallest delay.
Using an asymptotic framework which allows to capture that feature,
we study a class of appropriate sequential nonparametric kernel procedures 
under local nonparametric alternatives. 
We prove a new theorem on the convergence
of the normed delay of the associated sequential detection procedure
which holds for dependent time series under a weak mixing condition.
The result suggests a simple procedure to select a kernel from a 
finite set of candidate kernels, and therefore may also be of interest from a practical
point of view. Further,
we provide two new theorems about the existence and an explicit representation of 
optimal kernels minimizing the asymptotic normed delay.
The results are illustrated by some examples.\\ \vskip 2cm

\end{abstract}

\newpage
\section*{Introduction}

A classical problem of sequential analysis is to detect a location shift
in an univariate time series by a binary decision procedure.
More generally, we aim at testing sequentially whether the deterministic
drift function vanishes (in-control or null model) or is equal to an 
out-of-control or alternative model.
There are various important fields where such methods can be applied.
We shall first briefly describe some fields of applications which 
motivated the topics discussed in this article.

Sequential methods are applied for a long time in quality control
and statistical process control,
where interest focuses on detecting the first time point 
where a production process fails. Failures of a machine may produce
jumps in the sequence of the observed quality characteristic, whereas wastage
may result in smooth but possibly nonlinear changes of the mean. In recent years
there has been considerable interest in methods for dependent time series.

An active area is the on-line monitoring of sequential data streams from capital markets.
Indeed, an analysts task is to detect structural changes in financial data as soon as
possible in order to trigger actions as portfolio updates or hedges. Thus, methods designed
to support sequential decision making are in order.

A further potential field of application is the analysis of microarray time series data consisting
of gene expression levels of genes. Down- or upregulated genes can have important interpretations,
e.g., when characterizing cancer cells, and the sequential detection of such level changes
from time series could be of considerable value. 

In biology sequential methods may be useful to
study the temporal evolution of enzyme kinetics in order to detect time points
where a reaction starts or exceeds a prespecified threshold.
Often it is possible to associate certain (worst case) temporal patterns with
phenomena, e.g., symptoms or reactions to stimuli, which are of biological
interest to detect.
In order to understand complex biological systems it may be
useful to estimate such change points sequentially instead of applying
a posteriori methods, since the behavior of the real biological system 
depends only on the past. 

A basic model to capture level changes as motivated by the above application areas is as follows.
Suppose we are observing a possibly non-stationary stochastic
process, $ \{ \widetilde{Y}(t) : t \in \calT \} $, in continuous or discrete
time $ \calT = [0, \infty) $ with $ E | \widetilde{Y}(t) | < \infty $.
Consider the following decomposition of the process in a 
possibly non-homogenous drift $ m(t) = E \widetilde{Y}(t) $ and
an error process $ \{ \widetilde{\epsilon}_t : t \in \calT \} $,
$$
  \widetilde{Y}(t) = m(t) + \widetilde{\epsilon}_t, \qquad
  (t \in \calT).
$$
Using the terms of statistical process control, we will say that
the process is in-control, if $ m(t) = 0 $ for each $ t \in \calT $,
and we are interested in kernel control charts, i.e., 
sequential kernel-based smoothing methods, to detect
the first time point where the process gets out-of-control.
Clearly, from a testing point of view we are sequentially testing
the null hypothesis $ H_0: m = 0 $ against the alternative
$  m \not= 0 $. A common approach to the problem is to define
a stopping rule (stopping time), $ N $, based on some statistic 
that estimates at each
time point $ t $ a functional of $ \{ m(s) : s \le t \} $. Having defined
a stopping rule, the stochastic properties of the associated delay,
defined as $ N $ minus the change-point, are of interest.

Well-known stopping rules rely on CUSUM-, EWMA-, or Shewhart-type control
charts which can often be tuned for the problem at hand. 
These proposals are motivated by certain optimality criteria and have been
studied extensively in the literature. 
First publications are due to Page (1954, 1955), 
Girshick and Rubin (1952). Optimality properties of the CUSUM procedure
in the sense of Lorden (1971), i.e., minimizing the conditional expectation
of the delay given the least favorable event before the change-point
was first shown by Moustakides (1986) and, using Bayesian arguments, by
Ritov (1990, 1997) and Yakir (1997).
For a discussion of EWMA control schemes see Schmid and Schoene (1997).
Yakir, Krieger, and Pollak (1999) studied first order optimality of
the CUSUM and Shiryayev-Roberts procedures to detect a change in
regression. Their result
deals with optimal stopping rules in the sense that the expected
delay is minimal subject to a constraint on the average run length
to a false alarm. However, that result is restricted to independent
and normally distributed observations.
A kernel-based a posteori procedure for detecting multiple change points which is in the spirit
of the present article has been studied by Hu\v{s}kov\'a and Slaby (1997) and
Grabovsky, Horv\'ath and Hu\v{s}kov\'a (2000). For reviews we refer to 
Hu\v{s}kov\'a (1991) and Antoch, Hu\v{s}kov\'a, and Jaru\v{s}kov\'a (2002).

The present paper provides an asymptotic analysis with local alternatives,
which holds for a rich class of strongly mixing dependent processes.
Our stopping rule uses a Priestley-Chao type kernel regression estimate
which relies on a weighted sum of past observations without assuming knowledge
of the alternative regression function or an estimate of it.
The theory and application of such smoothing methods is nicely described in
Hart (1997) or H\"ardle (1990). However, it is important to note that 
framework and assumptions of the present paper are different from classical nonparametric regression.
Whereas in nonparametric regression it is assumed that the bandwidth $ h $ tends to $0$ such that
$ n h \to \infty $, $n$ denoting the (fixed) sample size, and $ \max t_i - t_{i-1} \to 0 $, as
$ n \to \infty $, our monitoring approach works with $ h \to \infty $ and 
$ t_i - t_{i-1} \ge \Delta $ for all $ i $.

Sequential smoothing procedures, where a regression estimate is evaluated at the
current observation, have been studied for various change-point problems,
e.g. to monitor the derivative of a process mean
(Schmid and Steland, 2000). Note also that they are implicitly applied in 
classical (fixed sample) nonparametric regression at the boundary.
Of course, it is of special interest to study the
simultaneous effect of both the kernel and the alternative drift 
on the asymptotic normed delay of the associated stopping rule. We prove a
limit theorem addressing this question for general mixing processes.
We then ask how to optimize the procedure w.r.t. the smoothing kernel
for certain regression alternatives. It turns out that an explicit
representation of the optimal kernel can be derived for arguments not
exceeding the associated asymptotic optimal delay.
For simple location shifts first results for the normed delay have been obtained
by Brodsky and Darkhovsky (1993, 2000) for sequential kernel smoothers as studied
here. When jumps are expected, jump-preserving estimators as discussed
in Lee (1983), Chiu et. al (1998), Rue et al. (2002), 
and Pawlak and Rafaj\l owicz (2000, 2001) are an attractive 
alternative, since smoothers tend to smooth away jumps.
Convergence results for the normed delay of jump-preserving stopping rules have been studied 
in Steland (2002a), where upper bounds for the asymptotic normed delay are
established. For a Bayesian view on the asymptotic normed delay and optimal prior choice see 
Steland (2002b). On-line monitoring has been recently reviewed by Antoch and Jaru\v{s}kov\'a (2002)
and Fris\'en (2003). We also refer to Siegmund (1985).

We shall now explain the asymptotic framework of our approach more detailed.
In order to evaluate a detection procedure we will consider
local alternatives which converge to the in-control model
as the (effective) sample size of the procedure tends to
infinity. Simultaneously, the false-alarm rate will tend to
$0$. The local nonparametric alternatives studied here are given
as a parameterized family of drift functions,
$$
  m(t) = m(t;h) = \eins(t \ge t_q^*) m_0( [t-t_q^*]/h ),
$$
where $ t_q^* \in \calT $ stands for the change-point
assumed to be fixed but unknown, 
and $ h > 0 $ is a bandwidth parameter of the detection 
procedure introduced below
determining the amount of past data used by the procedure.
We assume $ h \in \calH $ for some countable and unbounded set $ \calH \subset \R_0^+ $.
$ m_0 $ denotes the generic model alternative inducing the
sequence of local alternatives. We assume that $ m_0 $ is a piecewise Lipschitz
continuous function.
Our asymptotics will assume $ h \to \infty $.
Consequently, for each fixed $ t \in \calT $ we have
$ m_0(t;h) \to m_0(0) $, as $ h \to \infty $, if $ m_0 $ is
continuous in $ 0 $.
In this sense, $ m(t;h) $ defines a sequence of of local alternative
if $ m_0(0)=0 $. As we shall see below,
$ h $ coincides with the bandwidth parameter determining the sample size
of the kernel smoother
on which the sequential detection procedure is based on. It turns out
that the rate of convergence of the local alternative has to be related
to the bandwidth parameter in this fashion to obtain a meaningful
convergence result.

Assume the process is sampled at a sequence of fixed 
ordered time points, $ \{ t_n : n \in \N \} $, inducing a
sequence of observations $ \{ Y_n : n \in \N \} $.
Put $ m_{nh} = m(t_n;h) $ and 
$ \epsilon_n = \widetilde{\epsilon}(t_n) $ to obtain
$$
  Y_{nh} = m_{nh} + \epsilon_n, \qquad (n \in \N).
$$
Let $ q $ denote the integer ensuring $ t_q = \lfloor t_q^* \rfloor + 1 $. 
Then 
$$ 
  m_{nh} = \eins(t_n \ge t_q) m_0([t_n-t_q]/h), \qquad n \in \N.
$$
We do not assume that the time design $ \{ t_n \} $ becomes dense
in some sense. In contrary, we use time points having a fixed minimal
distance, and for simplicity we shall assume $ t_n = n $ 
for all $ n \in \N $. More general time designs will be discussed at the end
of Section~\ref{Asymptotics}.

The organization of the paper is as follows. Section~\ref{defs} 
provides basic notation, assumptions, and the definition of the
kernel detection procedure.
The limit theorem for the normed delay is established in
Section~\ref{Asymptotics}. The result holds for a wide
class of generic alternatives satisfying a mild integrability
condition, provided that the smoothing kernel is Lipschitz continuous.
Section~\ref{Optimal} provides the result on the optimal
kernel choice which minimizes the asymptotic normed delay. We provide both
an existence theorem and a stronger representation theorem.
Due to the close relationship of the optimal kernel and 
the generic alternative, this results requires both the
kernel and the regression alternative to be continuous.
We illustrate the results by a couple of examples.

\section{Sequential kernel detection and assumptions}
\label{defs}

We consider the following sequential Priestley-Chao type kernel smoother
\begin{equation}
\label{DetectionSmoother}
  \widehat{m}_{nh} = \sum_{i=1}^{n} K_h( t_i-t_n ) Y_{ih},
\end{equation}
$ n \in \N $. 
We call $ \widehat{m}_{nh} $ a {\em sequential} smoother, since at the $ n$-th
time point the Priestley-Chao type estimator $ t \mapsto \sum K_h( t_i - t ) Y_{ih} $ 
is only evaluated for $ t = t_n $. Here and in the sequel $ K_h(z) = K(z/h)/h $ denotes
the rescaled version of a smoothing kernel $ K $ required to be a centered, symmetric, and
Lipschitz continuous probability density. The associated $ \{ 0,1 \} $-valued 
sequential decision rule is given by
\begin{equation}
\label{DecisionRule}
  d_{nh} = \eins( | \widehat{m}_{nh} | > c ),
\end{equation}
or, $ d_{nh} = \eins( \widehat{m}_{nh} > c ) $ (one-sided version),
i.e., a signal is given if (the absolute value of) $ \widehat{m}_{nh} $
exceeds a prespecified non-negative threshold $ c $.

The corresponding stopping time is given by
$$
  N_h = \inf \{ n \in \N : d_{nh} = 1 \}
$$
with $ \inf \emptyset = \infty $. In addition, define the normed delay
$$
  \rho_h = \max \{ N_h - t_q, 0 \} / h.
$$
If the kernel vanishes outside the interval $ [-1,1] $,
the effective sample size of the detection procedure is equal to
$ h $. Then $ \rho_h $ is simply the delay expressed as a percentage
of the effective sample size.

In this paper we will measure the efficiency of a
decision procedure by the asymptotic behavior of its 
associated normed delay. We confine ourselves to stopping times meaning that
decisions at time $ n $ only depend on $ Y_1, \dots, Y_n $.

Throughout the paper we shall assume that
$ \{ \varepsilon_n \} $ is a stationary $ \alpha $ mixing process
in discrete time $ \N $. Recall that $ \alpha $ mixing (strongly mixing)
means that $ \alpha(k) \to 0 $, if $ k \to \infty $, 
where $ \alpha(k) $ denotes the $ \alpha $-mixing coefficient defined by
$$
  \alpha(k) = \sup_{A \in \calF_{-\infty}^0, B \in \calF_k^{\infty}}
  | P( A \cap B ) - P(A) P(B) |.
$$
Here $ \calF_{k}^{l} = \sigma( \varepsilon_k, \dots, \varepsilon_l ) $ stands for the
$ \sigma $-field induced by the random variables $ \varepsilon_k, \dots, \varepsilon_l $,
$ -\infty \le k \le l \le \infty $. Recall that $ \alpha $-mixing is a weak
notion of dependence which is implied by $ \beta $- and $ \rho$-mixing.
For a general discussion of mixing coefficients and related limit theorems we refer
to Bosq (1996). 
The regularity assumptions on the mixing coefficients of $ \{ \varepsilon_n \} $ will be given later.
In addition, we assume $ \{ \varepsilon_n \} $ satisfies Cramer's condition,
i.e.,
$$
  E e^{c_1 |\varepsilon_1|} < \infty
$$
for some positive constant $ c_1 $.

The smoothing kernel $ K $ used to define the weighting scheme is taken from the class 
$$ 
  \calK 
  = \left\{ K: \R \to [0,\infty) : \int K(s) \, ds = 1, K(s)=K(-s) \right\}
  \cap \mbox{Lip}
$$
of all symmetric probability densities on the real line which are Lipschitz
continuous, i.e., there exists a Lipschitz constant $ L_K $ ensuring
$$
  | K(z_1) - K(z_2) | \le L_K | z_1 - z_2 |, \qquad (z_1, z_2 \in \R).
$$
For our optimality results we will have to impose further conditions which will restrict the
class $ \calK $.

Finally, we also need the following conditions. It is assumed that 
$ m_0:[0,\infty) \to \R  $ is non-negative and satisfies, jointly with $ K \in \calK $,
the following integrability condition,
$$
  \biggl| \int_0^x K(s-x)m_0(s) \, ds  \biggr| < \infty \qquad (\forall x>0).
$$

\section{Asymptotics for the normed delay}
\label{Asymptotics}

In this section we establish both an assertion about the
in-control false-alarm rate and a limit theorem for the normed
delay for general local nonparametric alternatives under dependent sampling.

We need the following specialized 
large deviation result for the control statistic $ \widehat{m}_{nh} $.
A related large deviation result for (unweighted) sums of random variables satisfying Cramer's
condition can be found in Bosq (1996, Th. 1.4).
For our purposes we need the following specialised version for mixing time series.

Define $ S_{nh} = \sum_{i=1}^n K([t_i-t_n]/h) \epsilon_i $, $ n \in \N,$ $ h \in \calH $. 
For two real sequences $ (a_h) $ and $ (b_h) $ with $ b_h \not= 0 $ for sufficiently large $h $,
we write $ a_h \sim b_h $ if $ a_h/b_h \to 1 $, as $ h \to \infty $ and $ a_h \sim b_h $ up to a
constant if $ a_h / b_h \to c $, $ h \to \infty $, for some constant $ c $.

\begin{theorem} 
\label{MyLemma}
  Assume $ n/h \sim \zeta $ with $ 0 < \zeta < \infty $.
  Then the following assertions hold true.
  \begin{itemize}
    \item[(i)] 
      For each $ x > 0 $ 
      $$
        P( S_{nh} > xh ) =
       O\left( \frac{n}{\sqrt{h}} e^{-c_2 \sqrt{h} } \right) + O( n \alpha( \sqrt{h} ) ) + O( \sqrt{h} e^{-c_1 h } ) = o(1),
      $$
      as $ h \to \infty $, provided $ \lim_{k \to \infty} k^2 \alpha(k) = 0 $.
    \item[(ii)] If $ \sum_k k^2 \alpha(k) < \infty $, then for each $ x > 0 $
      $$ \sum_{h \in \calH} P( S_{nh} > xh ) < \infty $$ 
      implying 
      $$ P( S_{nh} > xh, i.o. ) = 0. $$
  \end{itemize}
\end{theorem}

\begin{remark}
  By construction of the stopping rule, Theorem~\ref{MyLemma} 
  also makes an assertion about the in-control false-alarm
  rate. Note that our setting implies that the rate
  converges to $ 0 $, as the effective sample size $ h $
  tends to infinity.
\end{remark}

\begin{proof} Put $ S_h = S_{nh} $.
  Fix $ 0 < \gamma < 1 $. Note that $ n/h \sim \zeta $.
  Partition the set $ \{ 1, \dots, n \} $ in blocks of length
  $ l(h) = \lfloor (\zeta h)^{1/2} \gamma \rfloor $ yielding
  $ b(h) = \lfloor n/l(h) \rfloor $ blocks. Note that 
  $ l(h) \sim h^{1/2} $ and $ b(h) \sim h^{1/2} $ up to constants.
  We have
  \begin{eqnarray*}
    S_{h}       & = & \sum_{r=1}^{l(h)} S_h^{(r)} + R_h \\
    S_{h}^{(r)} & = & \sum_{k=1}^{b(h)} K([t_{k b(h)+r}-t_n]/h) \epsilon_{k b(h)+r} \\
    R_{h}       & = & \sum_{i = b(h) l(h) + 1}^n K([t_i-t_n]/h) \epsilon_i.
  \end{eqnarray*}
  W.l.o.g. we can assume $ b(h)l(h) = n $, since
  $ P[ R_{h} > xh ] = O( l(h) e^{-c_1 h} ) $ for some constant $ c_1 > 0 $.
  Next observe that
  $$
    P[ S_{h} > xh ] \le \sum_{r=1}^{b(h)} P[ S_{h}^{(r)} > (xh)/b(h) ].
  $$
  Markov's inequality, Cramer's condition, the strong mixing property, and
  Volonski and Rosanov (1959)
  provide for each $ r = 1, \dots, l(h) $
  \begin{eqnarray*} 
  && \biggl|
    P \left[ S_{h}^{(r)} > \frac{xh}{b(h)} \right]
    -
    e^{-t(xh)/b(h)} \prod_{k=1}^{b(h)} E e^{t K([t_{k l(h)+r}-t_n]/h])\epsilon_k} 
    \biggr| \\
  && \quad \le 16(b(h)-1) \alpha( l(h) ).
  \end{eqnarray*}
  It is well-known that Cramer's condition holds iff. there are constants
  $ g > 0 $ and $ T > 0 $ such that
  $ E e^{t \epsilon_1} \le e^{gt^2} $ for all
  $ |t| \le T $ and $ g > (1/2) E \epsilon_1^2 $ (Petrov (1975), Lemma III.5). 
  Thus,
  \begin{eqnarray*}
    && e^{-tx} \prod_{k=1}^{b(h)} E e^{t K([t_{k l(h)+r}-t_n]/h)\epsilon_i}  \\
    && \qquad \le 
      \exp \biggl\{ gt^2 \sum_{k=1}^{b(h)} K([t_{k l(h)+r}-t_n]/h) - t xh /b(h) \biggr\}
  \end{eqnarray*}
  Minimizing the r.h.s. w.r.t. $ t $ gives the upper bound
  $$
   \left\{ \begin{array}{ll}
      e\left( - \frac{x^2 h^2}{b(h)^2} \frac{1}{2g C(r)} \right), &  (xh)/b(h) \le g T C(r) \\
      e\left( - \frac{xh}{b(h)} \frac{T}{2} \right),              &  (xh)/b(h)   > g T C(r) 
   \end{array} \right.
  $$
  where
  $ C(r) = \sum_{r=1}^{b(h)} K([t_{k l(h)+r}-t_n]/h)^2.$
  Observe that the timepoints 
  $$ t_{k l(h)+r}, k = 1, \dots, b(h) $$ 
  form
  an equidistant partition of an interval converging
  to $ (-\zeta,0) $. The size of the partition equals $ l(h)/h \sim h^{-1/2}$ up to a constant.
  Therefore, using $ \int_{-\zeta}^0 K(s)^2 \, ds = \int_0^{\zeta} K(s)^2 \, ds $,
  $$
  \biggl| \frac{l(h)}{h} \sum_{k=1}^{b(h)} K( [t_{k b(h)+r}-t_n]/h )^2
    - \int_0^{\zeta} K(s)^2 \, ds \biggr| = O\biggl( \frac{l(h)}{h} \biggr).
  $$
  Consequently, $ [l(h)/h)]^{-1} C(r) $ is bounded away from $0$ for large enough $h$.
  Thus, uniformly in  $ r = 1, \dots, b(h) $,
  $$
    P[ S_{h}^{(r)} > (xh)/b(h) ] = O( e^{-c h^{1/2}} ) + O( h^{1/2} \alpha( h^{1/2} ) ),
  $$  
  for some constant $ c > 0 $, yielding
  \begin{eqnarray*}
    P[S_{h} > xh] & \le & \sum_{r=1}^{b(h)} 
      P[ S_{h}^{(r)} > (xh)/b(h) ] + P[ R_h > x ] \\
   & = & 
      O( b(h) e^{-c h^{1/2}} ) 
    + O( b(h) h^{1/2} \alpha( h^{1/2} ) ) 
    + O( l(h) e^{-c_1 h } ) \\
   & = & 
      O\left( \frac{n}{\sqrt{h}} e^{-c_2 \sqrt{h} } \right) + O( n \alpha( \sqrt{h} ) ) + O( \sqrt{h} e^{-c_1 h } ).
  \end{eqnarray*}
  Therefore, the mixing condition
  $$
    \lim_{k \to \infty} k^2 \alpha(k) = 0 
  $$
  ensures 
  $$
    P[ S_{h} > xh ] = o(1), \quad \mbox{as $ h \to \infty$}.
  $$
  Finally, the above estimates and
  $$
    \sum_k k^2 \alpha(k) < \infty
  $$
  yield $ \sum_{h \in \calH} P[ S_{h} > xh ] < \infty $, and an application of Borel-Cantelli provides
  $$ P[ S_h > xh\ i.o. ] = 0. $$
\end{proof}

We may now formulate our main result on the strong law of large numbers for
the normed delay. Define
\begin{equation}
\label{Ch1:DefRho0Expl}
  \rho_0 = \inf
  \left\{ \rho > 0 : \int_0^{\rho} K(s-\rho) m_0(s) ds = c \right\}.
\end{equation}

\begin{theorem}
\label{MainResult}
  Let $ K \in \calK $ be a given kernel and $ m_0 $ be a piecewise Lipschitz continuous generic
  alternative with $ m_0(0) = 0 $ and either $ m_0 \ge 0 $ or $ m_0 \le 0 $
  such that (\ref{Ch1:DefRho0Expl}) exists and $ 0 < \rho_0 < \infty $.
  Then
  $$
    \rho_h \stackrel{a.s.}{\to} \rho_0,
  $$
  as $ h \to \infty $, provided that $ \sum_k k^4 \alpha(k) < \infty $.
\end{theorem}

\begin{remark} The proof even shows complete convergence.
\end{remark}

\begin{proof} W.l.o.g. we assume $ m_0 \ge 0 $.
  Let $ \varepsilon > 0 $. We shall estimate
  $ P[ \rho_h - \rho_0 > \varepsilon ] $ and
  $ P[ \rho_h < \rho_0 - \varepsilon ] $. 
  Put $ n(h) = \lfloor(\rho_0 + \varepsilon)h \rfloor $. 
  Then we have
  \begin{eqnarray*}
    P[ \rho_h - \rho_0 > \varepsilon ]
    & = & P[ N_h > (\rho_0 + \varepsilon) h ] \\
    & \le & P[ | \widehat{m}_{n(h),h} | \le c ] \\
    & = & P\left[ \left| \sum_{i=1}^{n(h)} 
      K_h( t_i - t_n )[m_{ih} + \epsilon_i] \right| \le c \right] \\
    & \le & 
    P\left[ \left| \sum_{i=1}^{n(h)} 
      K_h( t_i - t_n ) \epsilon_i \right| 
      > c - \left| \sum_{i=q}^{n(h)} K_h(t_i-t_n) m_{ih} \right| \right],
  \end{eqnarray*}
  where in the last step we used the fact that
  $ Y_{ih} = \varepsilon_i $ if $ 1 \le i < q = t_q $ and
  $ Y_{ih} = m_0([t_i-t_q]/h) + \epsilon_i $ if $ q \le i \le n(h) $.
  For the following argument we may assume that $ m_0 $ is Lipschitz continuous,
  since otherwise one may argue on subintervals. We have
  \begin{eqnarray*}
    \sum_{i=q}^{n(h)} K_h(t_i-t_n) m_{ih} 
      & = &
      h^{-1} \sum_{i=q}^{n(h)} K([i-n(h)]/h) m_0( [i-q]/h ) \\
      & = & 
      h^{-1} \sum_{i=q}^{n(h)} K(i/h - \rho_0 - \varepsilon ) m_0( i/h )+ O(h) \\
      &=&
      \int_0^{\rho_0+\varepsilon} K(s-\rho_0-\varepsilon) 
         m_0(s) ds + O(h^{-1}) \\
      & = & \int_0^{\rho_0} 
              K(s-\rho_0)m_0(s) ds + O(\varepsilon) + O(h^{-1}),
  \end{eqnarray*}
  since $ q/h \to 0 $ and $ n(h)/h \to \rho_0 + \varepsilon $, as
  $ h \to \infty $, and $ K $ is Lipschitz continuous.
  Recalling the definition of $ \rho_0 $, 
  there exists a constant $ \kappa > 0 $, which depends on $ \varepsilon $, with 
  $ c - | \sum_{i=q}^{n(h)} K_h(t_i-t_n) m_{ih} | \ge \kappa > 0 $, yielding
  $$
    P[ \rho_h - \rho_0 > \varepsilon ] \le P[ S_{n(h)} > \kappa h ].
  $$
  We may now apply Theorem~\ref{MyLemma} (ii) with $ \zeta = \rho_0 + \varepsilon $ 
  to conclude that 
  $$ 
    \sum_h P( \rho_h - \rho_0 >  \varepsilon) < \infty.
  $$
  To estimate $ P[ \rho_h < \rho_0 - \varepsilon ] $ note that
  $$
    \{ \rho_h < \rho_0 - \varepsilon \}
      = \{ q \le N_h \le q + \lfloor (\rho_0+\varepsilon)h \rfloor \},
  $$
  since $ N_h \ge q $ on $ \{ \rho_h < \rho_0 - \varepsilon \} $ by definition of $ \rho_h $.
  We have
  \begin{eqnarray*}
    && P[ q \le N_h \le q + \lfloor (\rho_0-\varepsilon) h \rfloor ]\\
      && \qquad \le 
      P\left[ \max_{q \le k \le q + \lfloor (\rho_0-\varepsilon)h \rfloor }
       | \widehat{m}_{kh} | > c \right] \\ \\
      && \qquad \le 
      \sum_{k=q}^{q + \lfloor (\rho_0-\varepsilon)h \rfloor}
        P\left[ \left| \sum_{i=1}^k K_h(t_i-t_k) \varepsilon_i  
         \right| > c - \left| \sum_{i=q}^{q + \lfloor (\rho_0-\varepsilon)h \rfloor}
         K_h( t_i - t_k ) m_{ih} \right| \right].
  \end{eqnarray*}
  First note that
  $$
    c - \sum_{i=1}^q K_h( t_i - t_k ) m_{ih} 
    \to \int_{\rho_0-\varepsilon}^{\rho_0} K(s-\rho_0) m_0(s) \, ds > 0.
  $$
  Further, $ q \le n \le q + \lfloor (\rho_0-\varepsilon)h \rfloor $
  implies $ n/h \to \rho_0 - \varepsilon $, as $ n,h \to \infty $.
  An application of Theorem~\ref{MyLemma} (i) to each summand with $ n = k $, noting that
  there are $ O(h) $ summands and, of course, $ k = O(h) $, we see that
  \begin{eqnarray*}
    P[ q \le N_h \le q + \lfloor (\rho_0-\varepsilon) h \rfloor ] 
      &=& O( h \sqrt{h} e^{-c_2 \sqrt{h} } ) + O( h^2 \alpha( \sqrt{h} ) ) + O( h \sqrt{h} e^{-c_1 h } ) \\
      &=& o(1),
  \end{eqnarray*}
  as $ h \to \infty $, provided $ k^4 \alpha(k) = o(1) $. Further,
  $
    \sum_k k^4 \alpha(k) < \infty
  $
  implies
  $$
    \sum_h P[ \rho_h - \rho_0 < - \varepsilon ] < \infty
  $$
  yielding complete convergence, 
  $$ 
    \sum_h P[ | \rho_h - \rho_0 | > \varepsilon ] < \infty, \quad \text{for every $ \varepsilon > 0 $},
  $$
  which implies a.s. convergence (e.g. Karr (1993), Prop. 5.7).
\end{proof}

We close this section with a brief discussion of more general time designs.
For some applications it may be possible and reasonable to determine at each time point 
the time points $ t_{n1}, \dots, t_{nn} $ where observations are taken.
For example, one may start with monthly observations and reduce the
distance between successive observations to ensure that the most recent data points are daily measurements.
Note that such a thinning effect can not be obtained by a smoothing kernel.

\begin{remark}
Assume $ F_T $ is a d.f. with support $ [0,1] $ possessing a density $ f_T $. Suppose at the $n$-th time
point we may select the time points where observations are taken. We assume that
$$
  t_{ni} = n F_T^{-1}(i/n), \qquad i = 1, \dots, n.
$$
Clearly, the choice $ t_{ni} = i $ corresponds to the uniform distribution $ F_T(s) = s, $ $ s \in [0,1] $. 
When using skewed time designs, we can ensure that more recent observations dominate the sample of size $n$.
It is straightforward to check that the proofs of Theorem~\ref{MyLemma} and Theorem~\ref{MainResult} also 
work for that choice of time points. In this case we obtain
\begin{eqnarray*}
  \sum_{i=1}^n K_h( t_{ni} - t_{nn} ) m( t_{ni}/h ) 
    & \to & \int_0^{\zeta} K( \zeta( F_T^{-1}(s/\zeta) - 1 ) ) m_0( F_T^{-1}(s/\zeta) ) \, ds \\
    & =   & \zeta \int_0^1 K( \zeta (s - 1) ) m_0(s) f_T(s) \, ds,
\end{eqnarray*}
if $ n/h \sim \zeta $, yielding
$$
  \rho_h \stackrel{a.s.}{\to} 
  \inf \left\{ \rho > 0 : \rho \int_0^1 K( \rho (s - 1) ) m_0(s) f_T(s) \, ds = c \right\},
$$
as $ h \to \infty $, under the conditions of Theorem~\ref{MainResult}.
\end{remark}

\section{Optimal kernels}
\label{Optimal}

The result of the previous section suggests the following kernel selection procedure.
Suppose we are given a finite set $ \{ K_1, \dots, K_M \} $ of candidate
kernels. Then we may choose the kernel which minimizes the corresponding
asymptotic normed delay $ \rho^* $. For an example where this selection rule was
successfully applied to a real data set see Steland (2002c).

However, the natural question arises how to optimize the
asymptotic normed delay with respect to the smoothing kernel $ K $.
It turns out that for the setting studied in this paper a meaningful
result can be obtained. The canonical solution of the functional optimization 
problem can be given explicitly for arguments not exceeding the asymptotic
optimal delay. Indeed, the optimal kernel
is equal to a composition of the generic alternative and 
a time-reversal transformation which depends on the optimal asymptotic
normed delay.

Although in this paper we assume that $m$ is continuous at $ t = t_q $, let us briefly
discuss the discontinuous classical change-point model given by 
$ m(s) = a\eins(s \ge t_q), a > 0 $ for all $ s \in \calT $.
Brodsky and Darkhovsky (1993, Th. 4.2.8) have shown
that the normed delay of the regular stopping rule 
(\ref{DecisionRule}) converges with probability 1, i.e., 
$
  \rho_h \stackrel{a.s.}{\to} \rho_0, 
$
as $ h \to \infty $, where the constant $ \rho_0 $ is given by
$$
  \rho_0 = \inf\left\{ \rho : \int_0^{\rho} K(s) ds = c/a \right\}.
$$
If $ F_{K}(z) = \int_{-\infty}^z K(s)\, ds,\ z \in \R, $ 
denotes the associated distribution function, we have the
explicit solution
$
  \rho_0(K) = F_{K}^{-1}( 1/2 - c/a ).
$
It is easy to show that for every $ c > 0 $ there exists a symmetric kernel
with unit variance and bounded support such that the functional $ \rho_0 $ vanishes.


Therefore, in the sequel we assume that $ m_0 $ is a non-constant function.
It will turn out that we now obtain solutions with non-vanishing optimal asymptotic normed
delay. This allows to define an ordering relation on the set of admissible
generic alternatives by comparing the optimal asymptotic normed delays.
Anticipating the relationship between the optimal
kernel and $ m_0 $, we assume that $ m_0 $ is Lipschitz continuous.
For the optimality result of this section we also have to assume the
following stronger regularity assumptions on the class of admissible kernels.
\begin{itemize}
  \item[(K1)] $ \calK $ is a class of uniformly Lipschitz continuous 
  probability densities with Lipschitz constant $ L $, i.e.,
  $$
    \sup_{K \in \calK} | K(z_1) - K(z_2) | \le L | z_1 - z_2 | \qquad (z_1,z_2 \in \R).
  $$
  \item[(K2)] The class $ \calK $ is uniformly bounded, i.e.,
  $$
    \| \calK \|_\infty = \sup_{K \in \calK} \| K \|_\infty \le C_{\calK} < \infty
  $$
  holds true for some constant $ C_{\calK} $.
\end{itemize}

Define the mapping $ I : \calK \times [0,\infty) \to \R $,
$$
  I( K, \rho ) = \int_0^{\rho} K(s-\rho) m_0(s) ds,
$$
and denote by
$$
  \calR( \rho ) = \{ I(K,\rho) : K \in \calK \}
$$
the {\em reachable set at time $ \rho $}. 
It is clear that $ \calR( \rho ) $
is closed when $ \calK $ is equipped with the uniform topology induced
by the supnorm.

\begin{definition}
A pair $ (K^*,\rho^*) \in \calK \times [0,\infty) $  
is {\em optimal}, if
$$
  \rho^* = \inf \{ \rho : c \in \calR( \rho ) \}
$$
and $ K^* $ ensures that
$$
  I( K^*, \rho^* ) = c.
$$
\end{definition}

For fixed $ K \in \calK $ define 
$$
  \Psi(\rho) = I(K,\rho) = \int_0^{\rho} K(s-\rho)m_0(s) \, ds.
$$
We will assume that there exists a positive $ R \in \R $ 
such that $ \Psi $ is a strictly increasing function on $ [0,R) $. 
Note that $ \Psi $ is continuous since $ K $ is Lipschitz continuous
by assumption. We have the following theorem on the existence of
optimal kernels.

\begin{theorem}
\label{ThEx}
Assume there exists $ K_1 \in \calK $ and some $ \rho_1 \ge 0 $ with
$$
  I( K_1, \rho_1 ) = \int_0^{\rho_1} K_1(s-\rho_1)m_0(s) ds = c.
$$
Then there exists an optimal kernel $ K^* \in \calK $, i.e.,
$$
  I( K^*, \rho^* ) = c
$$
where $ \rho^* = \inf\{ \rho : c \in \calR( \rho ) \} $ is the optimal
asymptotic normed delay.
\end{theorem}

\begin{remark} For many generic alternatives $ m_0 $ it should be a trivial task
   to verify the condition of Theorem~\ref{ThEx} holds true.
\end{remark}

\begin{proof}
By assumption we have $ c \in \calR( \rho_1 ) $. Let 
$ \rho^* = \inf\{ \rho : c \in \calR( \rho ) \} $. Then 
$ 0 \le \rho^* \le \rho_1 $. We shall show $ c \in \calR( \rho^* ) $.
Then, by definition of $ \calR( \rho^*) $, there exists an optimal kernel
$ K^* \in \calK $ with $ I( K^*, \rho^* ) = c $.
$ c \in \calR( \rho^* ) $ is a consequence of the following
continuity argument. There exists a non-increasing sequence
$ \{ \rho_n \} $ with $ \rho_n \to \rho^* $ and an associated
sequence $ \{ K_n \} \subset \calK $ with
$$
  I( K_n, \rho_n ) = c \in \calR( \rho_n ).
$$
Since $ K_n \in \calK $, $ I( K_n, \rho^* ) \in \calR( \rho^* ) $.
We have
\begin{eqnarray*}
  && | I( K_n, \rho^* ) - I( K_n, \rho_n ) | \\
  && \ = \left|
    \int_0^{\rho^*} K_n(s-\rho^*) m_0(s) ds - \int_0^{\rho_n} K_n(s-\rho_n) m_0(s) ds 
    \right| \\
  && \ =
  \biggl|
   \int_0^{\rho^*} [K_n(s-\rho^*) - K(s-\rho_n)] m_0(s) \, ds
   \\
  && \qquad -
   \int_{\rho^*}^{\rho_n} K_n(s - \rho_n)m_0(s) \, ds
  \biggr| \\
  && \ \le
  L | \rho^* - \rho_n | \int_0^{\rho^*} m_0(s) \, ds \\
  &&  \qquad
  + \left[ \int_{\rho^*}^{\rho_n} K_n^2(s-\rho_n) \, ds \right]^{1/2} 
    \left[ \int_{\rho^*}^{\rho_n} m_0^2(s) \, ds \right]^{1/2}.
\end{eqnarray*}
Clearly, by assumption (K2), 
$ \int_{\rho^*}^{\rho_n} K_n^2(s-\rho_n) \, ds \le C_{\calK}^2 | \rho^* - \rho_n | $.
Therefore, 
$$
  | I(K_n, \rho^*) - I(K_n,\rho_n) | = o(1),
$$
as $ n \to \infty $, yielding $ I(K_n,\rho^*) \to c $, as $ n \to \infty $.
Since $ I(K_n,\rho^*) \in \calR( \rho^* ) $ for all $ n $ and 
since $ \calR( \rho^* ) $ is closed, we obtain
$$
  c = \lim_{n \to \infty} I(K_n,\rho^*) \in \calR( \rho^* ).
$$
\end{proof}

The following Lemma provides an useful characterization of each
optimal pair $ (K^*,\rho^*) $ and is crucial to calculate optimal kernels.

\begin{lemma}
\label{OLemma}
Assume $ (\rho^*,K^*) \in [0,R] \times \calK $ is optimal. Then 
$$
  \int_0^{\rho^*} K^*(s-\rho^*) m_0(s) \, ds =
  \sup_{K \in \calK} \int_0^{\rho^*} K(\rho^*-s)m_0(s) \, ds
$$
\end{lemma}

\begin{proof}
Assume there exists some $ \widetilde{K} \in \calK $ with
$$
  \int_0^{\rho^*} \widetilde{K}(\rho^*-s) m_0(s) \, ds
  >
  \int_0^{\rho^*} K^*(\rho^*-s) m_0(s) \, ds = c.
$$
Since $ \rho \mapsto \int_0^{\rho} \widetilde{K}(\rho-s) m_0(s) \, ds $
is strictly increasing and continuous for $ \rho \in [0,\rho^*] $, there exists a $ \rho^{**} $
with $ \rho^{**} < \rho^* $ such that
$$ 
  \int_0^{\rho^{**}} \widetilde{K}(\rho^{**}-s)m_0(s) \, ds = c,
$$
implying that the pair $ (\rho^*, K^*) $ is not optimal which is a contradiction.
\end{proof}

We are now in a position to formulate and prove the following result about
the explicit representation of the optimal kernel for a given generic alternative.

\begin{theorem} 
\label{OptimalKernel}
In addition to the regularity assumptions of this section assume 
$$
  0 < \int_0^\infty m_0(s) \, ds < \infty
$$
and that the set
$$
  \calS = \left \{  \rho \ge 0 : 
    \frac{ \int_0^{\rho} m_0^2(s) \, ds}{ \int_0^{\infty} m_0(s) \, ds } = c
  \right\}
$$
is non-empty. Then the following conclusions hold true.
\begin{itemize}
\item[(i)] The optimal asymptotic normed delay is given by $ \rho^* = \inf \calS $
\item[(ii)] The optimal smoothing kernel $ K^* $ satisfies
$$
  K^*(z) = \frac{ m_0(\rho^*-|z|) }{ 2 \int_0^\infty m_0(s) \, ds }, \qquad
  z \in [-\rho^*, \rho^*].
$$
\end{itemize}
\end{theorem}

\begin{proof}
Let $ K \in \calK $ be an arbitrary candidate kernel.
By the Cauchy--Schwarz inequality we have
$$
  \left | \int_0^{\rho^*} K(\rho^* - s) m_0(s) \, ds \right|
  \le
  \left[ \int_0^{\rho^*} K^2(\rho^* - s) \, ds \right]^{1/2}
  \left[ \int_0^{\rho^*} m_0^2(s) \, ds \right]^{1/2}
$$
with equality if and only if
$$
  K( \rho^* - s ) = \lambda \cdot m_0(s), \qquad \forall s \in [0,\rho^*]
$$
for some constant $ \lambda \in \R $. 
Since $ \int_0^{\infty} K(s) ds = 1/2 $ and $ m_0(s) = 0 $ if $ s < 0 $,
$$
  \lambda^{-1} = 2 \int_0^{\infty} m_0(s) \, ds.
$$
Therefore, since $ \calK $ is a symmetric class,
$$
  K^*(s) = \frac{ m_0(\rho^*-|s|) }{ 2 \int_0^{\infty} m_0(s) \, ds }, \qquad
  s \in [-\rho^*,\rho^*].
$$
Finally, we obtain
$$
  c = \int_0^{\rho^*} K^*( \rho^* - s ) m_0(s) \, ds
    = \frac{ \int_0^{\rho^*} m_0^2(s) \, ds }{ 2 \int_0^{\infty} m_0(s) \, ds }.
$$
\end{proof}

\section{Examples} 

Let us consider some special cases to illustrate the results.
\begin{example}
For a truncated linear drift, 
  $$
    m_0(t) = a t \eins_{[0,T]}(t), \ t \ge 0,
  $$
  we obtain
  $
    \rho^* = \frac{3}{2} \frac{c}{a}
  $
  and 
  $$
    K^*(z) = \frac{1}{T^2} \left( \frac{3c}{2a} - |z| \right), 
    \qquad | z | \le \frac{3c}{2a}
  $$
\end{example}

\begin{example}
 Assume the generic alternative is given by a truncated exponential drift 
    $$ 
      m_0(t) = e^{\lambda t } \eins_{[0,T]}(t), t \ge 0, 
    $$ 
    for some $ \lambda \in \R $, where $ T $ is a positive truncation constant.
    If $ T \ge \rho^* $, we have
    $$
      c = \frac{ \int_0^{\rho^*} m_0^2(s) \, ds }
               { \int_0^{\rho^*} m_0(s) \, ds }
        = e^{\lambda \rho^*} + 1.
    $$
    Hence, the optimal asymptotic normed delay is given by
    $$
      \rho^* = \frac{\ln(c-1)}{\lambda}.
    $$
    The optimal kernel $ K^* $ is given by
    $$
      K^*(z) =  \frac{1}{2} \frac{e^{\lambda \rho^*}}{e^{\lambda T} - 1} \lambda e^{-\lambda |z|}.
    $$
    Note that $ K^* $ converges to the density of the Laplace
    distribution, $ (\lambda/2)e^{-\lambda |z|} $, 
    if $ \rho^* = T \to \infty $.
    Hence, for exponential drifts exponential weighting schemes are 
    asymptotically optimal in this sense.
\end{example}

\begin{example}
Usually, enzyme processes are described by the Michaelis-Menten 
  framework. Exploiting the quasi-steady-state approximations,
  the enzyme kinetic can be summarized by the differential equation
  $$
    \frac{d[S]}{dt} = -\frac{v_{\max} [S]}{K_M + [S]}
  $$
  with initial condition $ [S](0) = [S_0] $, 
  where $ [S](t) $ stands for the substrate concentration at time $t$,
  $ K_M $ denotes the Michaelis-Menten rate constant, and $ v_{\max} $ 
  is the maximal velocity. For further details we refer to Schnell and
  Mendoza (1997).
  The solution of the differential equation is given by
  $$
    [S](t) = K_M W\left( \frac{[S_0]}{K_M} \exp\left( \frac{-v_{\max} t 
    [S_0]}{K_M} \right) \right),
  $$
  where $ W $ stands for Euler's omega function, 
  the (principal branch of the) inverse of the function
  $ x \mapsto x \exp(x) $ (Euler (1777), Corless et al. (1996)).
  The optimal kernel to detect the generic alternative 
  $$ 
    m_0(t) = (S_0 - [S])(t) \eins_{[0,T]}(t) 
  $$ 
  is given by
  $$
   K^*(z) = (S_0-[S])( \rho^* - |z|) / C_S, \qquad |z| \le \rho^*.
  $$
  Observing that $ \int_a^b W( d e^\lambda t ) \, dt = \int_{de^{\lambda a}}^{de^{\lambda b}} W(y)y^{-1} \, dy $, 
  $ d, \lambda \in \R $, and using the formulas
  \begin{eqnarray*}
    \int_a^b W(y)/y \, dy   & = & \left. \frac{1}{6} W(y)^2[3+2W(y)] \right|_a^b, \\
    \int_a^b W(y)^2/y \, dy & = & \left. \frac{1}{2} W(y)[2+3W(y)] \right|_a^b,
  \end{eqnarray*}
  one may obtain explicit formulas for 
  $ C_S = 2 \int_0^T (S_0-[S])(t) \, dt $ and for the enumerator and denominator
  of the nonlinear equation
  $$
     c = \frac{ \int_0^{\rho^*} m_0^2(s) \, ds }{ \int_0^{\rho^*} m_0(s) \, ds }.
  $$
\end{example}

\section*{Acknowledgments}

The author is grateful to Prof. E. Rafaj\l owicz, Technical University
of Wroc\l aw, Poland, for helpful comments and a lecture on optimal kernel choice,
and two anonymous referees for their valuable remarks.

The financial support of the Deutsche Forschungsgemeinschaft (SFB 475
,,Reduction of complexity in multivariate data structures'') is gratefully
acknowledged.

\end{document}